\documentclass[a4paper,12pt]{article}     % `article' for A4 paper
\usepackage{amsmath,amscd,amssymb}        % AMS-LaTeX with CD and symbols
\usepackage{latexsym}                     % old LaTeX symbols
\usepackage[english, german]{babel}                % German & English (default)

\usepackage{theorem}                 % theorem extensions
\usepackage{color}

\newtheorem{theorem}{Theorem}
\newtheorem{corollary}{Corollary}

\theorembodyfont{\rmfamily}

\newtheorem{remark}{Remark}

\newenvironment{examples}
{\smallskip\noindent{\bf Examples\/}.}{\smallskip\par}
\newenvironment{proof}{\begin{ProofwCaption}{Proof}}{\end{ProofwCaption}}
\newenvironment{proof*}[1]{\begin{ProofwCaption}{{#1}}}{\end{ProofwCaption}}
\newenvironment{ProofwCaption}[1]%
  {\addvspace\theorempreskipamount \noindent{\it #1.}\rm}%
  {\qed \par \addvspace\theorempostskipamount}
\newcommand{\qedsymbol}{{\rm $\Box$}}
\newcommand{\qed}{\hfill\qedsymbol}
\newcommand{\ZZ}{{\mathbb Z}}

\newcommand{\calI}{{\cal I}}

\newcommand{\cc}{\underline{c}}

\newcommand{\kk}{\underline{k}}
\newcommand{\llll}{\underline{\ell}}
\newcommand{\ttt}{\underline{t}}
\newcommand{\mm}{\underline{m}}

\newcommand{\Var}{{\rm{Var}_{\mathbb{C}}}}
\newcommand{\Conf}{{\rm Conf}}
\newcommand{\UConf}{{\rm UConf}}

\title{Generating series of classes of exotic un-ordered configuration spaces}
\author{Sabir M.~Gusein-Zade
\thanks{The work 
was supported by the grant 21-11-00080 of the Russian Science Foundation.
}
}
\date{}

\begin{document}
\selectlanguage{english}

\maketitle

\begin{abstract}
A notion of exotic (ordered) configuration spaces
of points on a space $X$
was suggested by Yu.~Baryshnikov. He gave equations 
for the (exponential) generating series of the Euler characteristics of these spaces. Here we consider un-ordered analogues of these spaces. For $X$ being a complex quasiprojective variety, we give
equations for the generating series of classes of these configuration spaces in the Grothendieck ring
$K_0(\Var)$ of complex quasiprojective varieties. The answer is formulated in terms of the (natural) power structure over the ring $K_0(\Var)$. This
gives equations for the generating series of additive invariants of the configuration spaces such as the Hodge--Deligne polynomial and the Euler characteristic.
\end{abstract}

\section{Introduction}\label{sec:intro}
The configuration space of sequences of $k$
(different) points of a topological space $X$
(or the $k$th (ordered) configuration space of $X$) is
$$
\Conf(X,k)=\{(x_1,\ldots,x_k)\in X^k: x_i\ne x_j \mbox{ for } i\ne j\}\,,
$$
The configuration space of $k$-point subsets of a (topological) space $X$ (or the $k$th unordered configuration space of $X$) is 
$$
\UConf(X,k)=\{(x_1,\ldots,x_k)\in X^k: x_i\ne x_j \mbox{ for } i\ne j\}/S_k\,,
$$
where $S_k$ is the group of permutations on $k$ elements.
One has the well-known Macdonald type equation:
\begin{equation}\label{eqn:MacDonald_chi_simple_conf}
 1+\sum_{k=1}^{\infty} \chi(\Conf(X,k))\cdot t^k=(1+t)^{\chi(X)}\,.
\end{equation}
Here $\chi(\cdot)$ is the {\bf additive} Euler characteristic defined as the alternating sum of the ranks of the cohomology groups with compact support.
This Euler characteristic is the universal additive topological invariant (on some classes of topological spaces, say, on locally closed unions of cells in finite CW-complexes).
If one considers spaces with additional structure, one may have other addititve invariants: generalized Euler characteristics. For example, on the set of
complex quasiprojective varieties
the universal adiitive invariant is the class of a variety in the Grothendieck ring $K_0(\Var)$ of complex quasi-projective varieties. Other additive invariants are reductions of this universal one under (well-defined) group homomorphisms from $K_0(\Var)$ to the groups of values of the invariants. If an additive invariant is at the same time multiplicative, the corresponding homomorphism is a ring one. An invariant of this sort (additive and multiplicative) is the Hodge--Deligne polynomial of a complex quasiprojective variety. Thus it is obtained via a ring homomorphism from the
the Grothendieck ring $K_0(\Var)$ to the ring
$\ZZ[u,v]$ of polynomials in two variables.
An equation for the classes of complex quasiprojective varieties gives (through this homomorphism) an equation for their Hodge--Deligne polynomials.

A generalization of the notion of the configuration space $\Conf(X,k)$ called exotic
configuration spaces (see below) was suggested
by Yu.~Baryshnikov in~\cite{Baryshnikov}. Here we consider a natural version of this notion to the case of un-ordered subsets of points, i.e. a generalization of the notion of the unordered configuration space $\UConf(X,k)$.
We give Macdonald type equations for their Euler characteristics and (for the case when $X$ is a complex quasiprojective variety) their generalized versions for the generating series of classes of these configuration spaces in the Grothendieck ring $K_0(\Var)$. This gives a similar equatoion for the generating series of the Hodge--Deligne polynomials of these configuration spaces.

\section{The power structure over the Grothendieck
ring $K_0(\Var)$ and the Hodge--Deligne polynomial}\label{sec:prelim}
The notion of the power structure over a ring
was introduced in~\cite{GLM}. A power structure over a ring $R$ is a method to give sense to an expression of the form
$(1+a_1t+a_2t^2+\ldots)^m$ with $a_i$ and $m$ from $R$ as an element of
$1+b_1t+b_2t^2+\ldots\in 1+tR[[t]]$
satisfying some natural properties.
In~\cite{GLM} there was constructed a natural power structure over the Grothendieck ring $K_0(\Var)$ of complex quasiprojective varieties, i.e., for $a_i$ and $m$ being the classes $[A_i]$ and $[M]$ of complex quasiprojective varieties the coefficients $b_i$ were described as classes of quasiprojective varieties (or rather as sums of them) as well:
see Equation~(\ref{power-Grothendieck}) below

A power structure over a ring $R$ 
(defined as a map 
$(1+tR[[t]])\times R\to 1+tR[[t]]$)
permits to give sense to a similar expression
with the series depending on several variables $t_1$, \dots, $t_r$: see~\cite{GLM2}.

The ring $\ZZ$ of integers carries a natural power structure defined by the usual (say, in calculus) exponent of a series.

A complex quasiprojective variety is the difference of two complex
projective varieties (a projective variety minus a projective variety). The Grothendieck ring $K_0(\Var)$ of complex quasiprojective varieties is the abelian
group generated by the isomorphism classes $[X]$ of quasiprojective varieties modulo the relation
$[X]=[Y]-[X\setminus Y]$ for a Zariski closed subset $Y\subset X$. (The multiplication in $K_0(\Var)$ is defined by the Cartesian product.)
The (natural) power structure over the Grothendieck ring $K_0(\Var)$
is defined by the equation (see~\cite{GLM})

 \begin{eqnarray}\label{power-Grothendieck}
  \hspace{-20pt}&\ & (1+[A_1]t+[A_2]t^2+\ldots)^{[M]}=\nonumber\\
  \hspace{-20pt}&=&1+\sum_{k=1}^{\infty}\left(
  \sum_{\{k_i\}:\sum_i ik_i=k}
  \left[
  \left.
  \left(
  \left(
  M^{\sum_i k_i}\setminus \Delta
  \right)
  \times\prod_i A_i^{k_i}
  \right)
  \right/
  \prod_i S_{k_i}
  \right]
  \right)
  \cdot t^k\,,\label{power-K_0}
 \end{eqnarray}
where $A_i$, $i=1, 2, \ldots$, and $M$ are complex quasi-projective varieties, 
$[A_i]$ and $[M]$ are their classes in $K_0(\Var)$, $\Delta$ is the ``large diagonal'' in $M^{\sum_i k_i}$,
that is the set of (ordered) collections of $\sum_i k_i$ points from $M$ with at least
two coinciding ones, the group $S_{k_i}$ of permutations on $k_i$ elements acts by
simultaneous permutations on the components of the corresponding factor $M^{k_i}$ in
$M^{\sum_i k_i}=\prod_i M^{k_i}$ and on the components of the factor $A_i^{k_i}$.

An additive invariant of a complex quasiprojective variety $X$ is the Hodge--Deligne polynomial $e_X(u,v)\in \ZZ[u,v]$.
It is (completely) defined by the following two properties:
\begin{enumerate}
 \item[1)] for a smooth projective manifold $X$
 $$
 e_X(u,v)=\sum_{p,q} (-1)^{p+q}h^{p,q}(X)u^pv^q\,,
 $$
 where $h^{p,q}(X)$ are Hodge numbers of the manifold $X$;
 \item[2)] if $Y$ is a Zariski closed subvariety of $X$, one has $e_X(u,v)=e_Y(u,v)+e_{X\setminus Y}(u,v)$.
\end{enumerate}
The existence of the Hodge--Deligne polynomial was proved in~\cite{Deligne}.

A power structure over the ring $\ZZ[u,v]$ of polynomials in two variables was defined in~\cite{GLM-Mich} through a pre-$\lambda$-structure on this ring. For the case when all the coefficients $a_i$ are integers (that is polynomials do not depanding on $u$ and $v$; this is just the case in Section~\ref{sec:main}), the series $(1+a_1\ttt+a_2\ttt^2+\ldots)^{p(u,v)}$ ($a_i\in\ZZ$, $p(u,v)=\sum p_{ij}u^iv^j$) can be described in the following way.
The series $1+a_1\ttt+a_2\ttt^2+\ldots$ has a unique representation of the form
$\prod_{\mm\in\ZZ_{\ge0}^r\setminus{\overline{0}}}(1-\ttt^{\mm})^{s_{\mm}}$ with $s_{\mm}\in\ZZ$. Then
\begin{equation}\label{eqn:power_poly}
 (1+a_1\ttt+a_2\ttt^2+\ldots)^{p(u,v)}=
\prod_{\mm\in\ZZ_{\ge0}^r\setminus{\overline{0}}}\prod_{i,j}(1-u^iv^j\ttt^{\mm})^{-s_{\mm}p_{ij}}\,.
\end{equation}

The Hodge--Deligne polynomial determines a ring homomorphism $e_{\bullet}:K_0(\Var)\to\ZZ[u,v]$.
Moreover, $e_{\bullet}$ is a power structure homomorphism, that is
$$
e_{\bullet}\left((1+[A_1]t+[A_2]t^2+\cdots)^{[M]} \right)=\left(e_{\bullet}(1+[A_1]t+[A_2]t^2+\cdots)\right)^{e_M(u,v)}\,.
$$
(The homomorphism $e_{\bullet}$ acts on a series applying to its coefficients.)
This is a reformulation of a result from~\cite{Cheah}.

The Euler characteristic defines a power structure homomorphism $\chi:K_0(\Var)\to\ZZ$.

%%%%%%%%%%%%%%%%%%%%%%%%%%%%%%%%%%%%%
\section{Exotic configuration spaces}\label{sec:exotic_conf}
%%%%%%%%%%%%%%%%%%%%%%%%%%%%%%%%%%%%%
The notion of exotic (ordered) configuration spaces was introduced in~\cite{Baryshnikov}.

Assume that we consider $r$ different colors
numbered from $1$ to $r$.
Elements of $\ZZ_{\ge0}^r$ are called color counts. For an element $a=(c_1,\ldots, c_r)\in\ZZ_{\ge0}^r$, let $c_i(a):=c_i$. Let $\calI$ be a subset of $\ZZ_{\ge0}^r$
containing $\overline{0}$. (In~\cite{Baryshnikov} one considers only the case when $\calI$ is a so-called ideal in $\ZZ_{\ge0}^r$. In the discussion below this restriction will not be used.)
An $\calI$-permitted collection of colored points ($\calI$-collection for short) in $X$ is a finite subset $K$ of $X$ with a map $\psi:K\to \calI$. The (multi-)degree of a collection $(K,\psi)$ is
$\sum_{x\in K}\psi(x)$.
Let $\UConf_{\calI}(X,\kk)$ be the configuration space of $\calI$-collections in $X$ of multi-degree $\kk$. (If $\kk=\overline{0}$, $\UConf_{\calI}(X,\kk)$ is a one-point space.) If $X$ is a topological space, one has a natural topological structure on the parts of
$\UConf_{\calI}(X,\kk)$ corresponding to fixed collections $\{\psi(x)\vert x\in K\}$.
There are somewhat different possibilities
to define the topology on the union of this parts, i.e., on $\UConf_{\calI}(X,\kk)$,
however, the the Euler characteristic (being additive) does not depend on the choice of this topology. If $X$ is a complex quasiprojective variety, the configuration space 
$\UConf_{\calI}(X,\kk)$ in a natural way is the union of complex quasiprojective varieties.

\begin{examples}
 {\bf 1.} If $r=1$ and $\calI=\{0,1\}$, $\UConf_{\calI}(X,k)$ is the (unordered) configuration space $\UConf(X,k)$.
 \\
 {\bf 2.} If $r=1$ and $\calI=\ZZ_{\ge0}$, $\UConf_{\calI}(X,k)$ is the symmetric product $S^kX=X^k/S^k$.
 \\
 {\bf 3.} If $r=1$ and $\calI=\{k: 0\le k<m$,
 $\UConf_{\calI}(X,k)$ is the unordered no-$m$-equal configuration space in the sense of~\cite{Baryshnikov}. (For $m=2$ it coincides with the usual configuration space $\UConf(X,k)$ from Example~1.)
 \\
 {\bf 4.} If, for an arbitrary $r$, $\calI$ is the union of the coordinate axes that is consists of all
 points of $\ZZ^r_{\ge0}$ with not more than one coordinate different from zero (the “apartheid” ideal in the terms of~\cite{Baryshnikov}), the configuration space $\UConf_{\calI}(X,k)$ parametrizes collections of colored points where it is forbidden points
of different colors to collide, but the collision of any number of points of the same color is allowed.
\\
 {\bf 5.} If $\calI=\{\cc\in\ZZ^r_{\ge0}:c_1\le c_2\le\ldots\le c_r\}$, one can speak about the configuration space of nested subsets (a sort of a reduction of nested Hilbert schemes from~\cite{Cheah2}). 
\end{examples}

%%%%%%%%%%%%%%%%%%%%%%%%%%%%%%%%%%%%%%%%%%%%%%%%%%%%%%
\section{Generating series of classes of configuration spaces and of their invariants}\label{sec:main}
%%%%%%%%%%%%%%%%%%%%%%%%%%%%%%%%%%%%%%%%%%%%%%%%%%%%%%
Now we give equations for the generating series of the classes in $K_0(\Var)$ of the configuration
spaces $\UConf_{\calI}(X,\kk)$ (in the case when
$X$ is a complex quasiprojective variety) and of their invariants. For a subset
$\calI\subset \ZZ_{\ge 0}^r$, $\overline{0}\in\calI$, let 
$$
C_{\calI}(\ttt):=\sum_{\cc\in\calI}\ttt^{\cc}\in\ZZ[[t_1,\ldots, t_r]]\subset K_0(\Var)[[t_1,\ldots, t_r]]\,.
$$

\begin{theorem}\label{theo:main}
 For a quasiprojective variety $X$ one has
\begin{equation}
 \sum_{\kk\in\ZZ_{\ge0}^r}
 [\UConf_{\calI}(X,\kk)]\cdot\ttt^{\kk}=\left(C_{\calI}(\ttt)\right)^{[X]}\,.
\end{equation}
\end{theorem}

\begin{proof}
 The statement follows directly (more or less tautologically) from the following version
 of the description of the coefficients at the monomials $\ttt^{\kk}$ in
 \begin{equation}\label{eqn:power}
 \left(1+\sum_{\kk\in\ZZ_{\ge0}^r\setminus{\overline{0}}}[A_{\kk}]\cdot\ttt^{\kk}\right)^{[X]}
 \end{equation}
 ($A_{\kk}$ and $X$ are complex quasiprojective varieties)
 formulated in the classical case ($r=1$) by E.~Gorsky in~\cite{Gorsky}. Assume that,
 on the variety $X$, there live particles
 equiped with $r$ non-negative integer numbers
 (say, of different sorts of charges; colors in our case). A particle with charges $k_1$, \dots, $k_r$ has a space of internal states
 parametrized by points of a quasiprojective variety $A_{\kk}$ ($\kk=(k_1,\ldots, k_r)$). Then the coefficient at
 $\ttt^{\kk}$ in~(\ref{eqn:power}) is the configuration space of tuples of particles on $X$ with the total charges $k_1$, \dots, $k_r$.
 (In our case each of the varieties $A_{\kk}$ is either empty or consists of one point.)
 \end{proof}

 \begin{corollary}\label{cor:Hodge-Deligne}
 One has
 $$
  \sum_{\kk\in\ZZ_{\ge0}^r}
 e_{\UConf_{\calI}(X,\kk)}(u,v)\cdot\ttt^{\kk}=\left(C_{\calI}(\ttt)\right)^{e_X(u,v)}
 $$
 in the sense of the power structure over the polynomial ring
 $\ZZ[u,v]$ given by (\ref{eqn:power_poly}).
 \end{corollary}
 
 \begin{remark}
 In Examples~1-3 in Section~\ref{sec:exotic_conf} the series 
 $C_{\calI}(t)$ has the decomposition
 $(1-t)^{-1}(1-t^2)$, $(1-t)^{-1}$, and $(1-t)^{-1}(1-t^k)$ respectively. In Example~4 with
 $r=2$ one has 
 $C_{\calI}(t)=(1-t_1)^{-1}(1-t_2)^{-1}(1-t_1t_2)$.
 \end{remark}

 \begin{corollary}\label{cor:Euler}
 One has
 \begin{equation}\label{eqn:Euler_complex}
  \sum_{\kk\in\ZZ_{\ge0}^r}
 \chi(\UConf_{\calI}(X,\kk))\cdot\ttt^{\kk}=\left(C_{\calI}(\ttt)\right)^{\chi(X)}\,.
 \end{equation}
 \end{corollary}
 
 As a consequence of Theorem~{\ref{theo:main}},
 Corollary~\ref{cor:Euler} is valid for complex quasiprojective varieties. Let us show that Equation~(\ref{eqn:Euler_complex}) holds for a wider class of spaces, namely for spaces homeomorphic to locally closed unions of cells in finite CW-complexes.
 
\begin{theorem}\label{theo:main2}
 Let $X$ be a locally closed union of cells in a CW-complex $W$.
 Then one has
 \begin{equation}\label{eqn:Euler_CW}
  \sum_{\kk\in\ZZ_{\ge0}^r}
 \chi(\UConf_{\calI}(X,\kk))\cdot\ttt^{\kk}=\left(C_{\calI}(\ttt)\right)^{\chi(X)}\,.
 \end{equation}
\end{theorem}

\begin{proof}
Let us denote the left hand side of Equation~(\ref{eqn:Euler_CW}) by $\lambda_{X}^{\calI}(\ttt)$.
 Assume that $Y$ is also a union of cells of $W$, closed in $X$. Obviously one has
 $$
 \UConf_{\calI}(X,\kk)=\bigsqcup_{\llll\le\kk}
 \UConf_{\calI}(Y,\llll)\times
 \UConf_{\calI}(X\setminus Y,\kk-\llll)
 $$
 ($\llll=(\ell_1,\ldots,\ell_r)\le\kk=(k_1,\ldots,k_r)$ iff $\ell_i\le k_i$ for all $i$).
 This implies that 
 $$
 \lambda_{X}^{\calI}(\ttt)= \lambda_{Y}^{\calI}(\ttt) \lambda_{X\setminus Y}^{\calI}(\ttt)\,.
 $$
 Since $X$ is a disjoint union of (open) cells, it is sufficient to prove Equation~(\ref{eqn:Euler_CW}) for $X$ being a cell $\sigma^d$ of dimention $d$.
 One has $\chi(\sigma^d)=(-1)^d$.
 Thus it is necessary to show that
 \begin{equation}\label{eqn:cell}
 \lambda_{\sigma^d}^{\calI}(\ttt)=\left(C_{\calI}(\ttt)\right)^{(-1)^d}.
 \end{equation}
 For $d=0$ ($\sigma^0$ being a point) Equation~(\ref{eqn:cell}) is obvious.
 Assume that it is proved for cells of dimension
 less than $d$. In an obvious way the cell $\sigma^d$ can be decomposed into two $d$-dimensional and one $(d-1)$-dimensional cells.
 Therefore
 $$
 \lambda_{\sigma^d}^{\calI}(\ttt)
 =\left(\lambda_{\sigma^d}^{\calI}(\ttt)\right)^2 \lambda_{\sigma^{d-1}}^{\calI}(\ttt)
 $$
 and thus $\lambda_{\sigma^d}^{\calI}(\ttt)
 =\left(\lambda_{\sigma^d}^{\calI}(\ttt)\right)^{-1}$, what proves the statement.
\end{proof}

%%%%%%%%%%%%%%%%%%%%%%%%%%%%%%%%%%%%%%%%%%%

\medskip
\noindent Moscow State University, Faculty of Mechanics and Mathematics,\\
%% Moscow Center for Fundamental and Applied Mathematics,\\
Moscow, GSP-1, 119991, Russia\\
\& National Research University ``Higher School of Economics'',\\
Usacheva street 6, Moscow, 119048, Russia.\\
E-mail: sabir@mccme.ru

\end{document}